\documentclass{amsart}
\usepackage{pdfsync}
\usepackage{xspace}
\usepackage[all]{xy}
\usepackage{enumerate}
\usepackage{amssymb}
\usepackage{amsmath}

\makeatletter
\DeclareRobustCommand*\cal{\@fontswitch\relax\mathcal}
\makeatother

\newtheorem{definition}{Definition}

\newtheorem{thm}{Theorem}

\newtheorem{fact}[definition]{Fact}

\newcommand{\ann}{\mathop{{\rm ann}}\nolimits}

\newcommand{\mtx}{\mathfrak}

\newcommand{\RMod}{R\textrm{-Mod}}

\newcommand{\Ab}{\mathop{\rm Ab}}

\newcommand{\Hom}{\mathop{\rm Hom}}

\newcommand{\ppf}{{\rm ppf}}
\newcommand{\D}{{\rm D}}
\newcommand{\Lg}{{\rm L}}

\newcommand{\eq}{\,\dot=\,}
\newcommand{\di}{\,| \,}

\newcommand{\tr}{{\rm t}}

\newcommand{\ov}{\overline}

\newcommand{\br}{\ov}

\newcommand{\cB}{\cal B}
\newcommand{\cC}{\cal C}
\newcommand{\cF}{{\cal F}}

\newcommand{\T}{\mathop{\rm T}}

\newcommand{\eps}{\varepsilon}

\renewcommand{\phi}{\varphi}

\newcounter{numeq}
\newcounter{num}
\newcounter{one}
\newcounter{two}
\newcounter{three}
\newcounter{four}
\newcounter{five}
\newcounter{six}
\newcounter{seven}
\newcounter{eight}
\newcounter{nine}
\newcounter{ten}
\newcounter{eleven}
\newcounter{twelve}
\newcounter{thirteen}
\newcounter{fourteen}
\newcounter{fivteen}

\newcommand{\thenumm}{{\rm (\alph{num})}}
\newcommand{\thenummeq}{{\rm (\roman{numeq})}}

{\begin{list}{\thenumm}{\usecounter{num}}}%
{\end{list}}

{\begin{list}{\thenumm}{\usecounter{num}
\leftmargin0cm\itemsep0ex\parsep0ex\labelwidth0pt
\itemindent0.5em
\topsep0ex plus0ex minus0ex\parskip0ex}}%
{\end{list}}

{\begin{list}{\thenummeq}{\usecounter{numeq}\setlength{\leftmargin}{2em}
}}%
{\end{list}}

{\begin{list}{\thenummeq}{\usecounter{numeq}
\leftmargin1.1cm\itemindent0cm\itemsep0ex\parsep0ex
\topsep0ex plus0ex minus0ex\parskip0ex}}%
{\end{list}}

\begin{document}
\title{Implications of positive formulas in modules}
\author{Philipp Rothmaler}
\date{March 26, 2019}
\maketitle
{\centering\footnotesize To Miyuki\par}

\section{Introduction: implications} 
As is common in logic, given formulas $\phi$ and $\psi$, we write  $\phi\to\psi$ to mean the sentence $\forall \br x\, (\phi\to\psi)$, where $\bar x$ is a tuple containing all free variables of the formulas $\phi$ and $\psi$. We almost always suppress the universal quantifiers when writing implications.

We adopt the (not unusual) convention that disjunctions $\bigvee$ and conjunctions $\bigwedge$  bind more strongly than $\to$.

The most prominent kind of implication may be that of a basic Horn formula, that is an
implication $\chi$ of the form $\bigwedge\Phi\to \psi$, where $\Phi=\Phi(\br x)$ is a (possibly infinite) set of atomic formulas (in the same finitely many variables $\br x$) and $\psi=\psi(\br x)$ is a single atomic formula or $\perp$ (falsum).  According to \cite{H}, such formulas were shown by McKinsey in 1943 to be \emph{preserved in direct products}. (It was Horn in 1951, however, who proved the result in greater generality, whence the name.) That is, given a tuple $\br a$ of matching length in a direct product $A = \prod_I A_i$, if, for all $i\in I$, the coordinate tuple $\br a(i)$ of $\br a$ in $A_i$ satisfies $\chi$, then $\br a$ satisfies $\chi$ in $A$. The straightforward proof of this, cf.\ \cite[Thm.9.1.5]{H}, shows more.  

\begin{fact}\label{Horn}
If $\Phi$ is a set of formulas preserved  in direct factors and $\psi$ is a formula preserved in direct products, then the implication  $\bigwedge\Phi\to \psi$ is preserved in direct products.
\end{fact}

Here we say that $\Phi$ is \emph{preserved  in direct factors} if, conversely, the truth of $\chi(\br a)$ in $A$ entails that of $\chi(\br a(i))$ in $A_i$ for every $i\in I$ (and every possible direct product).

Recall that a \emph{primitive} formula is  an existentially quantified conjunction of atomic and negated atomic formulas. Thus a \emph{positive primitive} (henceforth \emph{pp})  formula is an existentially quantified conjunction of atomic  formulas. Clearly, pp formulas are preserved by homomorphisms. 

As is easily seen, pp formulas have both of the features figuring in the hypothesis of Fact \ref{Horn}, cf.\ \cite[Lemma 9.1.4]{H}.  Therefore  an implication of the form
$$\bigwedge_{\phi\in\Phi} \phi \to \psi,$$
with all $\phi$ and $\psi$ pp (and in the same variables again), is preserved in direct products. 
 
 Following  \cite{PRZ1}, we call such implications \emph{A-formulas}---for reasons that become clear when we look at  \emph{a}bsolutely pure modules below. (In fact, in that paper only \emph{A-sentences} were considered, by which we mean A-formulas  with all free variables universally quantified out.) 
 
 \begin{fact}\label{A} A-sentences are preserved in direct products.
 \end{fact}
Actually,  \cite{PRZ1} dealt only with modules (over an arbitrary associative ring with $1$), which is also the context of choice here. As atomic formulas are simply linear equations in this context, positive primitive formulas are existentially quantified finite systems of (homogeneous!) linear equations. In a module pp formulas thus define projections of solution sets of finite systems of linear equations (which always form additive subgroups). See \cite{P}  for more detail (or \cite{P2}, where they are called \emph{pp conditions}).

The simplest, and in some sense most prominent, examples of pp formula are $rx\eq 0$ and $\exists y (x\eq ry)$. The latter  we naturally write as $r\,|\,x$.


Simple examples like $nx\,\dot=\,0\to x\,\dot=\,0$  show that implications of pp formulas need not be preserved in epimorphic images and are therefore, by Lyndon's Theorem, in general not (equivalent to a) positive (formula). But they still behave decently enough. 

On a very general level this can be observed in positive logic as treated by Ben-Yaakov and Poizat \cite{BP}, where implications of finitary positive (but not necessarily primitive) formulas play an important role in their development of model theory---for instance, in the proof of the compactness theorem---for the simple reason that such formulas, they call them \emph{h-inductive}, are preserved in direct limits, as can easily be verified. This can be used to prove the compactness theorem for sets of h-inductive sentences in a rather straightforward way \cite[Cor.4]{BP}. (Moreover, they prove a preservation theorem saying, conversely,  that the theories preserved in direct limits are precisely those that can be axiomatized by h-inductive sentences \cite[Thm.23]{BP}.)

If now both antecedent and consequent in an h-inductive sentence are  (positive) \emph{primitive}, the implication is preserved in \emph{pure} epimorphic images (cf.\ \S\ref{purity}), which indicates how close such pp implications get to positive. 

 While also this makes sense for any similarity type of algebraic structures, \cite{purity}, it is modules where we know of most fruitful applications of this general model-theoretic viewpoint. And so, in this article we survey applications of positive implications (by which we really mean implications of positive formulas) in module categories. For reasons explained in \cite[Thm.4.1]{PRZ1}, all the implications we deal with are of the form  $\bigwedge\Phi\to \bigvee\Psi$, where both $\Phi$ and $\Psi$ are allowed to be infinite sets of pp formulas. Some features turn out nicer, cf.\ Fact \ref{Serre} below, if we `prepare' $\Psi$ to be closed under finite sums, so that $\bigvee\Psi$ becomes $\sum\Psi$ and the implication turns into $$\bigwedge\Phi\to \sum\Psi.$$ Following  \cite[Thm.4.1]{PRZ1}, we call such implications \emph{symmetric} sentences.

  We start, in \S\ref{alg}, with a list of such implications, some classical, some rather new, and go on, in \S\ref{log}, to exhibit consequences of the corresponding syntactic shapes.

%
%
%
This survey is an extended version of a talk on `Incomplete Theories' given at the model theory conference at RIMS, December 10--12, 2018. I wish to thank the organizer, Professor Hirotaka Kikyo, and RIMS for their hospitality and for providing the opportunity to present part of this material there.

\section{Preliminaries} 
%

\subsection{Tuples and matrices} Tuples are finite sequences that we think of as column or row vectors---depending on convenience. 
For instance, given two tuples $\br s$ and $\br t$ of same length, say $k$, we use $\br s\,\br t$ to denote the formal linear combination $\sum_{i<k} s_it_i$. I.e., we think of  matrix multiplication of the appropriate vectors. 
In other words, we think of $\br s$ as a row vector (in say $S$) and $\br t$ as column vector (in say $T$, assuming that there is a product defined between $S$ and $T$ and a sum between those products). This will always be clear from the context and it will always be assumed that the tuples are of the same length (but will never be mentioned again).

Most of the times, we have that $S$ is the ring $R$ and $T$ a left $R$-module. So typically $\br s$ would then be a row vector, while a tuple in a left module is usually thought of as a column vector.


\subsection{Modules and annihilators} Unless indicated otherwise, \emph{module} means left module over an associate ring $R$ with $1$. 
The \emph{annihilator} of $X\subseteq R$ in a module $M$ is the set $\ann_MX$ of all elements of $M$ that get annihilated by all elements of $X$. It is customary to write ${\mtx r}(r)$
instead of $\ann_{_RR}X$, also known as the \emph{right annihilator} of $X$ in $R$. (Then $\mtx l(X)$ would denote  the \emph{left annihilator}.)
 
 A \emph{domain} is a ring with no (nontrivial) zero divisors, i.e., a ring in which every nonzero element has trivial left and right annihilators (i.e., the annihilators are $0$ whether considered in $_RR$ or in $R_R$.)


\subsection{Languages and formulas} Some precaution is in order with the term \emph{first-order}, especially when infinitary languages are at play. The original sense of the term is that the variables can stand only for elements, and hence quantifications can only be over elements. (\emph{Second order} languages also have variables for sets of elements, and so on.) So whether infinitary or not, all expressions in $\Lg_{\infty \infty}$ are first-order. I will use the old term \emph{elementary} to single out the \emph{finitary} first-order part, i.e., $\Lg_{\omega \omega}$.   (For some reason, over the course of the last half century, model theorists have replaced `elementary' by `first order'---misleadingly so, as Tommy Kucera has pointed out to me.) The term \emph{pp formula} is reserved for existentially quantified \emph{finite} conjunctions of atomic formulas, which \emph{are} elementary.

The pp formulas of a given arity over a ring $R$---rather their classes modulo equivalence in \emph{all} $R$-modules---form a lattice with largest element $\br x \eq\br x$ and smallest element $\br x\eq \br 0$, where meet is conjunction and join is ordinary sum (of subgroups). To see that a conjunction of pp formulas is pp, simply pull out the quantifiers. 

For the sum one has to do some more rewriting. First of all, note that, given two pp formulas of same arity $\phi$ and $\psi$, in every module $M$, the sum $\phi(M) + \psi(M)$ can be defined by $\exists \br y \br z (\br x\eq \br y + \br z \, \wedge \phi(\br y) \wedge  \psi(\br z))$. Again, one can pull out all existential quantifiers to make it look pp.

\subsection{Purity}\label{purity} If the ring is a field, hence the modules are vector spaces, all embeddings split, i.e., every subspace is a direct summand. In general, direct summands may be  rare. But the concept of pure submodule plays an intermediate role in general module categories that salvages some of the convenient features of direct summands while being numerous enough.

For model theorists the easiest definition of purity (and one that applies to any algebraic structure) is via pp formulas: loosely speaking, an embedding of structures is \emph{pure} if it is an elementary embedding with only pp formulas under consideration. A \emph{pure substructure} is a substructure whose identical inclusion is a pure embedding. More precisely, a structure $A$ sitting inside a structure $B$  is a \emph{pure substructure} if $A\cap \phi(B) = \phi(A)$ for every pp formula $\phi$. (Pp formulas being existential, the inclusion from right to left is, of course, always true.) For modules, it suffices to consider unary pp formulas, see \cite[Prop.2.1.6]{P2}. It is easy to see upon projection, that direct summands are pure submodules (but the converse is far from true: e.g., every elementary substructure is pure). 

A morphism $g: B \to C$ is said to be a \emph{pure epimorphism}, \cite[\S1.6]{habil},  if for every pp formula $\phi$, every tuple in $\phi(C)$ has a preimage in $\phi(B)$. Again, for modules $1$-place $\phi$ suffice.  (Considering the trivial formula $x\eq x$, one sees that such a map is indeed surjective, thus justifying the term.)

Another convenient feature of modules is that in a short exact sequence $0\to A \to B\to C \to 0$, the monomorphism is pure  if and only if the epimorphism is,  \cite[Cor.3.5]{habil}. In this case we speak of a \emph{pure-exact sequence}.

Lazard proved that pure-exact sequences are precisely the direct limits of split exact sequences, which once again points to their importance (cf.\ \cite[Prop.2.1.4]{P2} and the references given there).

\subsection{Elementary (Prest--Herzog) duality}\label{D} Mike Prest \cite{P} found an an antiisomorphism $\D$, called \emph{elementary duality},  between the lattice of pp formulas (of a given arity) for left modules and the lattice of pp formulas (of same arity) for right modules. Let's look at it in arity $1$. Then $\D$ sends $r|x$ to $xs\eq 0$ and $sx\eq 0$ to $s|x$ (on the other side, i.e., to $\exists y(x\eq ys)$). In particular, $x\eq x$ gets sent to $x\eq 0$ and vice versa. Also, meet goes to join and join goes to meet, so that $\D(\phi\wedge\psi) = \D\phi + \D\psi$ and $\D(\phi + \psi) = \D\phi \wedge \D\psi$.
For the entire definition, see \cite{P2}, \cite{Her}, or \cite{habil}.

Ivo Herzog \cite{Her} extended elementary duality to theories (which is why it is also called \emph{Prest-Herzog duality}). This makes it an even more powerful tool. In particular, the dual of an implication $\phi\to\psi$ is defined to be  $\D\psi\to\D\phi$.

Finally, \cite{PRZ1} extended all this to certain infinitary formulas. It suffices to know here that the dual of a symmetric sentence $\bigwedge\Phi\to \sum\Psi$ is defined to be $\bigwedge\D\Psi\to \sum\D\Phi,$ which is symmetric again. (Here by the dual of a set I mean the set of the duals.)

An extremely useful application of elementary duality is 

\begin{fact}[Herzog's criterion]\label{crit}

$\br a \otimes \br b = 0$ in a tensor product $A \otimes_R B$ if and only if  there is a pp formula $\phi$ (of matching arity) such that $\br a \in \D\phi(A)$ and $\br b \in \phi(B)$.
\end{fact}

Here we adopt the same convention about tuples: by $\br a \otimes \br b$ we mean the linear combination of simple tensors, $\sum_{i} a_i\otimes b_i$.

%


\section{Algebra}\label{alg}
\subsection{Classical torsion} Just as for abelian groups, a left module over a domain $R$ is called \emph{torsionfree} if it satisfies all implications $rx\eq0\to x\eq0$, where $r$ runs over the nonzero ring elements. 
So the axioms are a set of implications of atomic formulas. As the consequent is always the same, one can arrange the entire set into a single---possibly infinitary---implication of positive formulas, namely, $$(\bigvee_{0\not= r\in R}rx\eq0)\to x\eq0.$$ Dually, a left $R$-module is called \emph{torsion} (or \emph{periodic} in group theory) if it satisfies the implication $$x\eq x\to(\bigvee_{0\not= r\in R}rx\eq0).$$

In the consequent of this implication one can replace disjunction by sum, for, over a commutative domain,   $(rx\eq 0 + sx\eq 0) \to (rs)x\eq  0$ (which shows that these axioms are F-sentences in the sense of \S\ref{F}). 

We observe a curious asymmetry: while torsionfreeness can be axiomatized by a (possibly infinite) set of elementary sentences, namely pp implications, being torsion is not in general an elementary property---take, for example, the (torsion) Pr\"ufer group $\Bbb Z_{p^\infty}$, which is elementarily equivalent to $\Bbb Z_{p^\infty}\oplus \Bbb Q$, which is mixed.

\subsection{Torsion theory} 
This  is done in torsion theory as follows, cf.\ \cite{St}. One first defines a map $\T$ from $\RMod$, the category of left $R$-modules to the category of sets  by $\T(M)=\{m\in M\,:\,rm\,=\,0 \text{ for some } 0\not=r\in R\}$, which is easily seen to be a functor. As the so-called \emph{torsion part} $\T(M)$ forms a subgroup of the underlying additive group of $M$, it is in fact a functor from $\RMod$ to the category $\Ab$ of abelian groups. Now, it is not hard to see that $\T(M)$ is even a submodule, making the map $\T$ what is called a \emph{preradical}. 

Given any such preradical, $\tr$, one says that $M$ is \emph{torsionfree} (with respect to $\tr$) if $\tr(M) = 0$, and \emph{torsion} if  if $\tr(M) = M$.  If the preradical $\tr$ happens to be definable, these classes are axiomatized by the implications
$$\tr(x) \to x\eq 0,$$ 
respectively,
$$x\eq x \to \tr(x).$$

The preradical $\T$ enjoys the extra property that $\T(M/\T(M))=0$, which makes it a \emph{radical}. This is the same as to say that the factor module of a module modulo its torsion part is torsionfree.

The important feature of this radical for us is that it is definable by an infinite disjunction, $\bigvee_{0\not= r\in R}rx\eq0$, so that both extreme properties of torsion and torsionfree become axiomatized by implications involving that disjunction. We will encounter more such definable radicals.

That the map $\T$ is a radical depends heavily on properties of the ring. The proof that it is is easy when $R$ is a commutative domain. Ore introduced conditions that make sure the same works even when the domain is no longer commutative, cf.\ \cite{St}. However, there are domains that are not Ore. 

Another issue is to extend this to non-domains. The problem being that the ring---as a, say left, module over itself (denoted $_RR$)---is no longer $\T$-torsionfree if it has zero divisors.  Namely, any left zero divisor is a member of $\T(_RR)$. One does not like that. The remedy is to let $r$ run  only over so-called \emph{regular} elements (which means non-zero divisors, on either side). If one now puts the Ore conditions on those, one gets the same effect for that adjusted radical. 

A \emph{torsion theory} in the technical sense is then the pair of the two classes of torsionfree modules and of torsion modules. Every such pair gives rise to a preradical, whose value on a module $M$ is defined to be the maximal torsion submodule of $M$, see \cite[Ch.VI]{St}.

\subsection{Hattori torsion}\label{Hatt} An elegant solution to the aforementioned problems with zero divisors was found by Hattori \cite{Hat}. He defined a module $F$ to be \emph{torsionfree}, we will say \emph{$\mtx h$-torsionfree}  (where $\mtx h$ stands for Hattori),\footnote{In the literature, these are often simply called \emph{torsionfree}, cf.\ e.g.\ \cite[\S 2.3.2]{P2}.} if an element in $F$ can be annihilated by a ring element $r\in R$ only if it is a linear combination of elements in which every coefficient itself is annihilated by $r$. This can be expressed by infinitary implications as follows.

    Given a tuple $\br s$ in $R$, let $\br s\, |\,x$ denote the pp formula $\exists\, \br y\, (x\eq \br s\,\br y)$ (remember, with this notation the tuples are assumed to be of matching length).
 
 For the purposes of the next implication, given $r\in R$, write    
   $\mtx v(r)$ for the set  $\mtx r(r)^{<\omega}$ of all (finite) row vectors  $\br s$ with entries from $\mtx r(r)$. In other words,  $\mtx v(r)$ is the set of all row vectors from $R$ with $r\,\br s\,=\,\br 0$.
    

Then the above statement about annihilation and linear combinations  can be expressed by the (possibly infinitary) implication $$r\,x\eq 0\,\,\,\to \bigvee_{\br s\in \mtx v(r)} \br s\,|\,x.$$ Therefore $\mtx h$-torsionfreeness is axiomatized by all such implications where $r$ runs over all (sic!) of $R$.

Again, torsion theory (as an algebraic theory) has a standard way of obtaining a torsion theory (in the technical sense) from a notion of torsionfreeness. Define $T$ to be \emph{$\mtx h$-torsion} if $\Hom(T, F)=0$ for every $\mtx h$-torsionfree module $F$. Let $\mtx h$ be the functor from $\RMod$ to $\Ab$ that singles out, in any $M\in\RMod$, the largest $\mtx h$-torsion submodule $\mtx h(M)$. Hattori \cite{Hat} shows 
that this is always a preradical and gives conditions when it is a radical, see also \cite[\S 5]{Tf} for a  discussion.

Note that $\mtx h$-torsionfreeness becomes classical torsionfreeness when the ring is a domain (commutative or not).
\subsection{Flat modules} Note that the disjunction $\bigvee_{\br s\in \mtx v(r)} \br s\,|\,x$ in Hattori torsion defines in $_RR$ exactly the right annihilator of $r$. To make things more visible, let $\phi_r$ stand for  the pp formula $r\,x\eq 0$ defining it. 

As  mentioned, in $_RR$, the disjunction $\bigvee_{\br s\in \mtx v(r)} \br s\,|\,x$ simply defines $\phi_r(_RR)$. But in an arbitrary module $M$ it defines 
$ \phi_r(_RR) M$, that is, the additive subgroup of $M$ generated by all products $r\,m$ with $r\in \phi_r(_RR)$ and $m\in M$, hence exactly the group of all linear combinations in $M$ with coefficients in $\phi_r(_RR)$. In other words, $M$ satisfies the implication $r\,x\eq 0\,\,\,\to \bigvee_{\br s\in \mtx v(r)} \br s\,|\,x$ precisely when $\phi_r(M)\subseteq \phi_r(_RR) M$. 

So $M$ is $\mtx h$-torsionfree if and only if it satisfies these implications for all $r\in R$, if and only if $\phi_r(M)\subseteq \phi_r(_RR) M$ for all $r\in R$. 
(Note, the reverse implication is always true, in any module.)

Requiring this inclusion $$\phi(F)\subseteq \phi(_RR) F$$
 for every pp formula $\phi$ one obtains the notion of \emph{flat} module $F$---by a result  of Zimmermann \cite{Zim}, see \cite{P} and  \cite[Thm.2.3.9]{P2}; but for the purposes at hand, we may adopt this  as the definition. (And again, it is easily seen that the inclusion from right to left is always true.)
 
 From this one sees at once that every flat module is $\mtx h$-torsionfree. Over a commutative principal ideal domain, in particular, for abelian groups, the converse is true. More generally, the same holds for RD-rings, see \cite[2.4.16]{P2}.

To produce an axiomatization of flatness by infinitary implications, first 
generalize the previous notation to an arbitrary pp formula $\phi$ by writing    
   $\mtx v(\phi)$ for the set  $\mtx \phi(_RR)^{<\omega}$ of all (finite) row vectors  $\br s$ with entries from $\phi(_RR)$ and consider the
 infinite disjunction $\bigvee_{\br s\in \mtx v(\phi)} \br s\,|\,x,$ which indeed defines 
 $\phi(_RR) M$ in every module $M$. 

It is now clear that a module is flat if and only if it satisfies 
  all the implications 
$$\phi\to\bigvee_{\br s\in \mtx v(\phi)} \br s\,|\,x,$$
where $\phi$ runs over all pp formulas (and again, $1$-place $\phi$ turn out to suffice).

The union  the disjunction  defines in a module  is exactly the sum of the subgroups defined by the disjuncts (for torsionfreeness, simply concatenate the tuples $\br s$, but this is as clear for arbitrary $\phi$, since, adding two disjuncts, one simply gets longer linear combinations of the same kind). Thus we finally obtain an axiomatization of flatness by the symmetric sentences 
$$\phi\to\sum_{\br s\in \mtx v(\phi)} \br s\,|\,x$$
for every (unary) pp formula $\phi$.

\subsection{Purely generated modules}\label{puregen}
Here I discuss an axiomatization result from \cite{HR}, which generalizes that for flat modules and for which some more  background is needed.

It can be easily understood that in the ring $R$, as a left module over itself, every element $r$ not only satisfies the pp formula 
$r|x$, but that this is the smallest pp formula it satisfies, i.e., for every unary pp formula $\phi$ with $r\in \phi(_RR)$, one has $r|x\leq\phi$. We say, $r|x$ \emph{generates} the pp type of $r$ in $_RR$. Clearly, this passes on to direct powers of  $_RR$, that is, to \emph{free} modules, except one needs to allow divisibility formulas $\br r|x$ with tuples as divisors, as before. And finally it passes down to direct summmands of free modules, which is to say, to \emph{projective} modules. This can be done for tuples as well, and one obtains that pp types in projective modules are generated by divisibility formulas, \cite[Lemma 1.2.29]{P2}. 

More can be said. If $\br a$ is a tuple in a projective module $P$ whose pp type is generated by such a pp formula $\phi$, then this tuple actually satisfies this formula freely: we say $(P, \br a)$  is a \emph{free realization} of $\phi$ if $\br a$ satisfies (or realizes) $\phi$ in $P$ and, whenever a tuple $\br b$ in a module $M$ satiesfies $\phi$ as well, then there is a homomorphism $f: P \to M$ sending $\br a$ to $\br b$. It is an easy exercise to show, this implies that $\phi$ generates the pp type of $\br a$ in $P$.

Modules in which every tuple freely realizes some divisibility formula are called \emph{locally projective}. If we allow arbitrary pp formulas as generators, we obtain the concept of \emph{locally pure projective} or \emph{strict Mittag-Leffler} module, see \cite{sML} for a detailed discussion.

We say a class $\cC$ is \emph{purely generated} by a class $\cB$ if every member of $\cC$ is a pure-epimorphic image of a direct sum of modules from $\cB$. This concept plays an important role. For instance, by a theorem of Lenzing, the flat modules are purely generated by finitely generated projectives, cf.\ discussion in \cite{HR} before Thm.\ 2.1.

Suppose $\cB$ is a class of locally pure projective modules closed under finite direct sum, and $\cC$ is the class purely generated by $\cB$. It was shown in \cite[Thm.2.1]{HR} that then $\cC$ is axiomatized by implications of the form
$$\phi \to \sum \ppf_{\cB}\phi,$$
where $\ppf_{\cB}\phi$ is the (usually infinite) set of pp formulas below $\phi$ (in the lattice order) that are freely realized in some member of $\cB$. 

In \cite{sML}, more such axiomatizations can be found. (Note, a proper definition of $\ppf_{\cB}\phi$ is missing in \cite{HR}, but it was alluded to at the end of \cite[Rem.4.2]{HR}.)

\subsection{Injective torsion} Returning to torsion, consider injective torsion as introduced in \cite[Def.\,2.1]{MR}, whose radical  $\mtx s$ is defined by letting the torsion part $\mtx s(M)$ of $M\in\RMod$ be the kernel of the map $\eps\otimes_R 1_M$, where $\eps$ denotes an injective envelope $R_R \to E$ of $R_R$. Applying elementary duality (\S \ref{D}), especially Herzog's criterion Fact \ref{crit},  one can quite easily show, \cite{MaRo}, that $\mtx s$ is definable by a (possibly infinite) disjunction of pp formulas. More precisely, $\mtx s(M)$ is the union of all pp subgroups $\psi(M)$ for which $\psi(_RR)=0$,  \cite{MaRo}. This shows that the \emph{$\mtx s$-torsion} modules are axiomatized by the implication 
$$x\eq x\to \mtx s(x),$$
 while the \emph{$\mtx s$-torsionfree} modules are axiomatized by the implication 
 $$\mtx s(x)\to x\eq 0.$$ 
 Here $\mtx s(x)$ denotes the aforementioned disjunction of formulas. While this may involve an infinite disjunction in the antecedent, as with classical torsion the entire  implication is equivalent to the set of pp implications
 $$\psi(x) \to x\eq 0,$$
 where $\psi$ runs over the pp formulas that vanish in $_RR$. Thus the class of $\mtx s$-torsionfree modules is axiomatized by pp implications (and thus elementary). More will be said in \S \ref{defsub}.

\subsection{Positive primitive (pp) torsion} In direct generalization of the behavior of injective torsion, one may consider the following generalization (introduced in \cite{Tf}). Given any class $\cF$ of left $R$-modules, let $\mtx s_\cF(x)$ be the disjunction of all pp formulas that vanish on all members of $\cF$. Denote by $\mtx s_\cF$ the map from $\RMod$ to sets that singles out in $M\in\RMod$ the subset defined by the disjunction
$\mtx s_\cF(x)$. It is quite easy to see \cite[Rem.5.5]{Tf} that the map $\mtx s_\cF$ is a preradical---namely, in every module $M$, the union defined by the disjunction in question is in fact a sum, and, moreover, it forms a submodule of $M$. The interesting new fact is that it always is
a radical \cite{MaRo}. And so we have two new classes, the \emph{$\mtx s_\cF$-torsion} modules and the \emph{$\mtx s_\cF$-torsionfree} modules, axiomatized, respectively, by the implications
$$x\eq x\to \mtx s_\cF(x)$$
and 
$$\mtx s_\cF(x)\to x\eq 0.$$ 

For the same reason as in injective torsion, also the class of $\mtx s_\cF$-torsionfree modules is elementary.

Note, $\mtx s=\mtx s_{_RR}=\mtx s_\cF$ for $\cF= {_R\flat}$, the class of flat left $R$-modules (the latter equality follows from the axiomatization of flat modules  above).

\subsection{Classical divisibility} Let $R$ be a domain. The elementary dual of the torsionfree axiom $rx\eq0\to x\eq0$ is $x\eq x \to r|x$. Doing this for all nonzero ring elements, we see that the elementary dual of the class of torsionfree left $R$-modules is that of \emph{divisible} right $R$-modules. Similarly we see that the class of divisible left $R$-modules is axiomatized by the set of implications $$x\eq x \to r|x$$ with $0\not= r\in R$.

\subsection{Hattori divisibility}\label{div} Taking the (infinitary) elementary dual of  Hattori's torsionfreeness (or $\mtx h$-torsionfreeness) of \S\ref{Hatt}, one obtains Hattori's definition of divisibility: a module over an arbitrary ring $R$ is called \emph{divisible} if it satisfies the following (in general, infinitary) implication $$(\bigwedge_{sr=0, s\in R} sx\eq 0) \to r|x.$$

Written in plain language, this means that a module is divisible iff an element $x$ is divisible by $r\in R$ whenever the left annihilator of $r$ in $R$ annihilates also $x$ (which is certainly necessary for  $x$ to be divisible by $r$). It is curious to note that Lam in \cite[Def.3.16]{L} makes precisely this definition without quoting Hattori's work and without making the dual definition of torsionfreeness.

\subsection{Absolutely pure modules}\label{abs}
A module is called \emph{absolutely pure} if it is pure in every module that contains it as a submodule. Since direct summands are pure, an \emph{injective} module (i.e., a module that constitutes a direct summand in in every module that contains it) are absolutely pure. (Over noetherian rings, the converse is true, so one may just think of injectives.) It was shown in \cite[Prop.1.3]{PRZ2} (see also \cite[Prop.2.3.3]{P2}) that $M$ is absolutely pure if and only if every pp subgroup is an annihilator, more precisely, iff for every pp formula $\phi$ one has $\phi(M)=\ann_M \D\phi(R_R)$. Note, this makes sense, as $\D\phi$ is a right formula and can thus be applied to the ring as a right module.

It follows from the easy direction of Herzog's criterion Fact \ref{crit} that the inclusion from left to right is always true. Recall from \S \ref{purity} that it suffices to consider $1$-place pp formulas for purity. The same applies to absolute purity. One concludes that the class of absolutely pure modules is axiomatized by  the implications 
$$({\bigwedge_{r\in \D\phi(R_R)} rx\eq 0})\,\,\,\,\,\,\,\,  \to \,\,\,\,\phi(x),$$
one for every $\phi$.

These axioms are dual, under the extended elementary duality of \cite{PRZ1}, to the axioms of flatness for right modules, see the introduction to \S4 in that paper. This extends an earlier result of Herzog \cite{Her}, showing the same for the case that these classes are elementary (which, by a result of Eklof and Sabbagh, is the case precisely when $R$ is left coherent).

\section{Logic}\label{log}
Ever since the beginning of model theory (or universal algebra for that matter)---namely since Birkhoff's preservation theorem---it has been known that the mere syntactic shape of axioms may have sweeping consequences on the structure of its model class.  We illustrate this  by the examples that we have encountered in the previous section. Most of this holds true in the larger context of structures in any algebraic signature.

\subsection{Valid pp implications: Lemma Presta} 

It would be a shame to talk about implications of pp formulas in modules without mentioning a description of those that are
\emph{valid}, i.e., true in all left $R$-modules---in which case we write $\phi\leq\psi$. This description was found by Mike Prest \cite[\S 8.3]{P} and brought to my attention by Gena Puninski about three decades ago as `Lemma Presta.' I have always found the exact statement difficult to memorize, which is why I wish to present it here---in a way that is easy to  remember.

Lemma Presta, \cite[Lemma 8.10]{P} or  \cite[Cor.1.1.16]{P2},
%
 describes the pp implications $\phi\to\psi$ that are valid (in $\RMod$), i.e., for which we have $\phi\leq\psi$.  It shows  that for any pp  formulas $\phi$ and $\psi$, one can find a first order sentence in the language of rings whose truth in $R$ is equivalent to $\phi\leq\psi$. 
 Here is how. Let $\phi$  and $\psi$ be written in matrix form as $A\di B\br x$ and $C\di D\br x$, respectively. 
 

Write the matrices $A, B, C$  with another three unknown matrices $X, Y, Z$ as \\ 
$X, \hspace{2mm} A, \hspace{2mm}  B, \hspace{2mm}  C, \hspace{2mm}  Y, \hspace{2mm}  Z$ \hspace{1mm}  and form the  products \hspace{1mm} $XA, \hspace{2mm} XB, \hspace{2mm} CY, \hspace{2mm} CZ$. Assuming that all the matrices match appropriately, state  that the outer two products are equal, $XA\eq CZ$, and that $D$ is equal to the sum of the inner two,  $XB + CY \eq D$.  It is an easy exercise in formalization to write the statement that is intended to  mean $\exists\, X\,  Y\,  Z\, (XA\eq CZ\,\wedge\, XB + CY \eq D)$
as a first order statement in the theory of rings. Let us prove that this is the desired sentence.
 
%
%
\begin{thm} [Lemma Presta]\text{} 

 $A\di B\br x\,\leq\,C\di D\br x$ if and only if $R\models \exists\, X\,  Y\,  Z\,  (XA\eq CZ\,\wedge\, XB + CY \eq D)$.
\end{thm}

\begin{proof}

If the right hand side holds, we have 

$A\di B\br x\,\leq\,XA\di XB\,\br x\, \sim CZ\di XB\,\br x\,\leq\,C\di XB\,\br x$, and since $C\di CY\,\br x$ is trivially true, also $A\di B\br x\,\leq\,C\di (XB + CY)\,\br x$, which is the left hand side, as desired.

For the hard direction, consider a module $M$ with generators $\br a\, \br b$ and relations $A\br a = B\br b$, or 
  $\begin{pmatrix}
A,&-B\\
\end{pmatrix} \begin{pmatrix}
           \br a \\
           \br b
         \end{pmatrix}= \br 0$. We will use the standard feature of a module given by generators $\br z$ and relations $G\br z=\br 0$ (where $G$ is a matching matrix over the ring) that any other matching matrix $H$ annihilating $\br z$ must be divisible on the right by $G$, i.e., there must be a matching matrix $X$ with $H=XG$.

         Note that $\br b$ satisfies $A\di B\br x$ in $M$, hence, by hypothesis, also $C\di D\br x$. Pick a witness $\br c$ in $M$ such that $C\br c = D\br b$. Now, being in $M$ allows us to write $\br c$ as $Z\br a + Y\br b$, so $CZ\br a + CY\br b= D\br b$, hence 
$\begin{pmatrix}
CZ,&CY-D\\
\end{pmatrix}  \begin{pmatrix}
           \br a \\
           \br b
         \end{pmatrix}=\br 0$. The standard feature mentioned before yields a matrix $X$ such that $\begin{pmatrix}
CZ,&CY-D\\
\end{pmatrix} = X \begin{pmatrix}
A,&-B\\
\end{pmatrix} $. Consequently, $XA = CZ$ and $XB + CY = D$, as desired.
   \end{proof}
   
This result is an essential ingredient in the development of elementary duality, cf.\ \cite[Thm.8.21]{P}, \cite[Prop.1.3.1]{P2}, or \cite[Prop.1.10]{habil}. See \cite[Thm.39]{Her2} for Lemma Presta viewed as a completeness theorem.

\subsection{Symmetric sentences} 
The following is easily  verified using the very definition of purity (that pp formulas pass up and down in extensions).

\begin{fact}
 Symmetric sentences are preserved in pure substructures. 
\end{fact}

\subsection{A-sentences} Symmetric sentences in which the the consequent consists of a single pp formula we called A-sentences and saw, in Fact \ref{A}, that A-sentences are preserved in direct products. (We are back to infinitary implications: $\Phi$ is allowed to be an infinite set of pp formulas.) 

Remembering the axiomatizations of divisibility (\S \ref{div}) and absolute purity (\S \ref{abs}) by (possibly infinitary) A-sentences, one infers that the classes of divisible modules and that of absolutely pure modules are closed under direct products. This implies the same  for the class of injectives modules, for `injective = pure-injective + absolutely pure' and direct products of pure-injective modules are pure-injective.

The question arises how much more can be said: can one extend this to arbitrary symmetric sentences? The answer is `no' by a variant of McKinsey's lemma, see \cite[Cor.9.1.7]{H}, whose straightforward proof works just as well for pp formulas.

\begin{fact}[McKinsey's lemma]\label{McK} Let $\Phi$ and $\Psi$ be sets of pp formulas.

If $\bigwedge \Phi\to \bigvee \Psi$ is preserved in direct products, then there is a single $\psi\in\Psi$ such that $\bigwedge \Phi\to \bigvee \Psi$ is equivalent to $\bigwedge \Phi\to \psi$. 

In particular, symmetric sentences closed under product are A-sentences.
 \end{fact}

\subsection{Pp implications: definable subcategories}\label{defsub} Consider an A-sentence as above. If $\Phi$ is finite (or, equivalently, a singleton\footnote{Finite conjunctions of pp formulas are pp.}), we can write it as a \emph{pp implication}, i.e., an implication of the form $\phi\to\psi$ with both $\phi$ and $\psi$ pp formulas. Having characterized, in Lemma Presta, the pp implications that are always true, we go on to describe (classes axiomatized by) arbitrary pp implications.

Such  sentences are elementary (finitary first-order), and therefore classes of models of pp implications are elementary classes. Crawley-Boevey called such classes, or rather the corresponding full subcategories of modules,  \emph{definable subcategories} (of $\RMod$).\footnote{For a discussion of the history of (and the names for) that concept, see \cite[\S 1.1]{HR}.} Since trivially pp implications are  A-sentences, definable subcategories  are elementary classes closed under direct products.

The largest definable subcategory of $\RMod$ is $\RMod$ itself. This is, because it even is  an \emph{equational class} (or \emph{variety} in Birkhoff's terminology of universal algebra), i.e., a class defined by sentences of the form $\exists \br x \alpha$ with $\alpha=\alpha(\br x)$ a term equation (hence an atomic formula). Note, such a sentence is equivalent to the implication $\br x\eq\br x\to \alpha$, clearly a pp implication.

There is a preservation theorem for definable subcategories: a class of modules constitutes a definable subcategory if and only if it is closed under direct products, direct limits and pure substructures, \cite{purity} or \cite[Thm.3.4.7]{P2}. From this it follows that definable subcategories are exactly the (full subcategories on) \emph{additive elementary classes}, i.e., elementary classes closed under finite direct sum and direct summand, see \cite[\S 1.1]{HR} again.

Plain inspection of the axioms shows that among the examples discussed in the algebra section above, the class of torsionfree modules over a domain and the classes of $\mtx s_\cF$-torsionfree modules (with $\mtx s_\cF$ pp torsion)   constitute definable subcategories. In particular, the $\mtx s$-torsionfree modules (for injective torsion $\mtx s$) do.

\subsection{F-sentences: Pure Serre subcategories}\label{F}\text{}
Following \cite[\S4.2]{PRZ1}, an \emph{F-sentence} is a symmetric sentence with antecedent a singleton, i.e., an implication of the form $\phi \to \sum\Psi$ with $\phi$ pp and $\Psi$ a set of pp formulas. An \emph{F-class} is a class of structures axiomatized by F-sentences.


A class $K$ is said to be a \emph{pure Serre subcategory} if, in every pure-exact sequence, the outer terms belong to $K$ if and only if  the middle term does. In other words, pure Serre subcategories are closed under pure submodules, pure epimorphic images and pure extensions. 

It is easy to see that F-sentences are preserved in pure submodules and pure epimorphic images. For instance, given a pure epimorphism $g: B \to C$ such that $B\models \phi \to \bigvee\Psi$, if $\br c\in \phi(C)$, then, by purity, there's a preimage $\br b\in\phi(B)$, which then must satisfy also some $\psi\in\Psi$; finally, any homomorphism preserves pp formulas, hence $\br c\in \psi(C)$, which concludes the proof that $C\models \phi \to \bigvee\Psi$. 

The converse is also true:

\begin{fact}\label{Serre} \cite[Lemma 2.2]{HR}  F-classes form pure Serre subcategories. 
\end{fact}

But to show that F-sentences are preserved in pure extensions, though still straightforward, one needs to make sure that the big disjunction is in fact a sum, i.e., one needs to have the implication in true F-form $\phi \to \sum\Psi$, see the details in the proof of \cite[Lemma 2.2]{HR}. The same applies to the next property of F-classes---here one has to add disjuncts for every coordinate in the (finite!) support of an element in the direct sum.
\begin{fact}\label{sum}\cite[Prop.\ 2.3]{HR} F-classes are closed under direct sum.
\end{fact}

In arbitrary direct products the  proof breaks down as one cannot expect to be able to add up infinitely many disjuncts from $\Psi$  to get a single one. In fact,  McKinsey's lemma, Fact \ref{McK}, applies and shows that, under the assumption that the F-sentence is preserved in direct products, a single disjunct must have worked throughout  to begin with: 

\begin{fact}
 If an implication $\phi \to \bigvee\Psi$ is preserved in direct products, then there is a single $\psi\in\Psi$ such that $\phi \to \bigvee\Psi$ is equivalent to $\phi \to \psi$. 
 
 In particular, F-sentences preserved in direct products are pp implications. Hence F-classes closed under direct product are definable subcategories.
\end{fact}

(This was extended to pure Serre subcategories in  \cite[Prop.\ 2.8]{HR}.)

Examples of F-classes are all classes of torsion modules for the various kinds of torsion discussed in the algebra section (except for possibly Hattori torsion). Further, the classes of Hattori torsionfree modules, of flat modules and the ones from \S\ref{puregen}
are also F-classes.
%
%
%
%
%
%
%
%
%

\end{document}